\documentclass[11pt]{amsart}
\usepackage{anysize} \marginsize{1in}{1in}{1in}{1in}

\usepackage{xcolor}
\usepackage{amsmath}

\makeatletter
\newcommand*\rel@kern[1]{\kern#1\dimexpr\macc@kerna}
\newcommand*\widebar[1]{%
  \begingroup
  \def\mathaccent##1##2{%
    \rel@kern{0.8}%
    \overline{\rel@kern{-0.8}\macc@nucleus\rel@kern{0.2}}%
    \rel@kern{-0.2}%
  }%
  \macc@depth\@ne
  \let\math@bgroup\@empty \let\math@egroup\macc@set@skewchar
  \mathsurround\z@ \frozen@everymath{\mathgroup\macc@group\relax}%
  \macc@set@skewchar\relax
  \let\mathaccentV\macc@nested@a
  \macc@nested@a\relax111{#1}%
  \endgroup
}
\makeatother

\usepackage{mathtools}
\usepackage[all]{xy}
\usepackage[utf8]{inputenc}
\usepackage{varioref}
\usepackage{amsfonts}
\usepackage{amssymb}
\usepackage{bbm}
\usepackage{xparse}

\usepackage{graphicx}
\usepackage{mathrsfs}

\usepackage{tikz,tikz-cd}
\usetikzlibrary{matrix,arrows}
\makeatletter
\tikzset{
  edge node/.code={%
      \expandafter\def\expandafter\tikz@tonodes\expandafter{\tikz@tonodes #1}}}
\makeatother
\tikzset{
  subseteq/.style={
    draw=none,
    edge node={node [sloped, allow upside down, auto=false]{$\subseteq$}}},
  Subseteq/.style={
    draw=none,
    every to/.append style={
      edge node={node [sloped, allow upside down, auto=false]{$\subseteq$}}}
  }
}

\usepackage[OT2,T1]{fontenc}
\DeclareSymbolFont{cyrletters}{OT2}{wncyr}{m}{n}
\DeclareMathSymbol{\Sha}{\mathalpha}{cyrletters}{"58}

\usepackage[hypertexnames=false,backref=page,pdftex,
 	pdfpagemode=UseNone,
 	breaklinks=true,
 	extension=pdf,
 	colorlinks=true,
 	linkcolor=blue,
 	citecolor=red,
 	urlcolor=blue,
 ]{hyperref}

\DeclareUnicodeCharacter{2212}{-}
\DeclareUnicodeCharacter{02BC}{'}

\let\oldeqref\eqref
\makeatletter
\RenewDocumentCommand\eqref{s m}{%
  \IfBooleanTF#1%
  {\textup{\tagform@{\ref*{#2}}}}
  {\oldeqref{#2}}
}
\makeatother

\newcommand{\sE}{{\mathcal E}}
\newcommand{\sF}{{\mathcal F}}

\newcommand{\sI}{{\mathcal I}}

\newcommand{\sO}{{\mathcal O}}

\newcommand{\Z}{{\mathbb Z}}

\newcommand{\End}{{\operatorname{End}}}
\newcommand{\Ext}{{\operatorname{Ext}}}

\newcommand{\Hom}{\operatorname{Hom}}

\renewcommand{\iff}{\Leftrightarrow}

\newcommand{\Pic}{\operatorname{Pic}}

\newcommand{\rk}{{\rm rk}}

\newcommand{\supp}{\operatorname{supp}}

\renewcommand{\to}[1][]{\xrightarrow{\ #1\ }}

\makeatletter
\newcommand*{\da@rightarrow}{\mathchar"0\hexnumber@\symAMSa 4B }
\newcommand*{\da@leftarrow}{\mathchar"0\hexnumber@\symAMSa 4C }
\newcommand*{\xdashrightarrow}[2][]{%
  \mathrel{%
    \mathpalette{\da@xarrow{#1}{#2}{}\da@rightarrow{\,}{}}{}%
  }%
}
\newcommand{\xdashleftarrow}[2][]{%
  \mathrel{%
    \mathpalette{\da@xarrow{#1}{#2}\da@leftarrow{}{}{\,}}{}%
  }%
}
\newcommand*{\da@xarrow}[7]{%
  \sbox0{$\ifx#7\scriptstyle\scriptscriptstyle\else\scriptstyle\fi#5#1#6\m@th$}%
  \sbox2{$\ifx#7\scriptstyle\scriptscriptstyle\else\scriptstyle\fi#5#2#6\m@th$}%
  \sbox4{$#7\dabar@\m@th$}%
  \dimen@=\wd0 %
  \ifdim\wd2 >\dimen@
    \dimen@=\wd2 %
  \fi
  \count@=2 %
  \def\da@bars{\dabar@\dabar@}%
  \@whiledim\count@\wd4<\dimen@\do{%
    \advance\count@\@ne
    \expandafter\def\expandafter\da@bars\expandafter{%
      \da@bars
      \dabar@ 
    }%
  }%
  \mathrel{#3}%
  \mathrel{%
    \mathop{\da@bars}\limits
    \ifx\\#1\\%
    \else
      _{\copy0}%
    \fi
    \ifx\\#2\\%
    \else
      ^{\copy2}%
    \fi
  }%
  \mathrel{#4}%
}
\makeatother

\newtheoremstyle{citing}
  {}
  {}
  {\itshape}
  {}
  {\bfseries}
  {\textbf{.}}
  {.5em}
  {\thmnote{#3}}

\theoremstyle{plain}

\newtheorem{theorem}[subsection]{Theorem}

\newtheorem{lemma}[subsection]{Lemma}
\newtheorem{corollary}[subsection]{Corollary}

\newtheorem{proposition}[subsection]{Proposition}

\theoremstyle{remark}
\newtheorem{example}[subsection]{Example}

\theoremstyle{definition}
\newtheorem{conjecture}[subsection]{Conjecture}
\newtheorem{definition}[subsection]{Definition}

\numberwithin{equation}{section}

\theoremstyle{remark}
\newtheorem{remark}[subsection]{Remark}

\theoremstyle{citing}

\makeatletter
\newsavebox\myboxA
\newsavebox\myboxB
\newlength\mylenA

\newcommand*\xtilde[2][0.8]{%
    \sbox{\myboxA}{$\m@th#2$}%
    \setbox\myboxB\null
    \ht\myboxB=\ht\myboxA%
    \dp\myboxB=\dp\myboxA%
    \wd\myboxB=#1\wd\myboxA
    \sbox\myboxB{$\m@th\widetilde{\copy\myboxB}$}
    \setlength\mylenA{\the\wd\myboxA}
    \addtolength\mylenA{-\the\wd\myboxB}%
    \ifdim\wd\myboxB<\wd\myboxA%
       \rlap{\hskip 0.5\mylenA\usebox\myboxB}{\usebox\myboxA}%
    \else
        \hskip -0.5\mylenA\rlap{\usebox\myboxA}{\hskip 0.5\mylenA\usebox\myboxB}%
    \fi}

\newbox\usefulbox

\def\getslant #1{\strip@pt\fontdimen1 #1}

\def\xxtilde #1{\mathchoice
 {{\setbox\usefulbox=\hbox{$\m@th\displaystyle #1$}%
    \dimen@ \getslant\the\textfont\symletters \ht\usefulbox
    \divide\dimen@ \tw@ 
    \kern\dimen@ 
    \xtilde{\kern-\dimen@ \box\usefulbox\kern\dimen@ }\kern-\dimen@ }}
 {{\setbox\usefulbox=\hbox{$\m@th\textstyle #1$}%
    \dimen@ \getslant\the\textfont\symletters \ht\usefulbox
    \divide\dimen@ \tw@ 
    \kern\dimen@ 
    \xtilde{\kern-\dimen@ \box\usefulbox\kern\dimen@ }\kern-\dimen@ }}
 {{\setbox\usefulbox=\hbox{$\m@th\scriptstyle #1$}%
    \dimen@ \getslant\the\scriptfont\symletters \ht\usefulbox
    \divide\dimen@ \tw@ 
    \kern\dimen@ 
    \xtilde{\kern-\dimen@ \box\usefulbox\kern\dimen@ }\kern-\dimen@ }}
 {{\setbox\usefulbox=\hbox{$\m@th\scriptscriptstyle #1$}%
    \dimen@ \getslant\the\scriptscriptfont\symletters \ht\usefulbox
    \divide\dimen@ \tw@ 
    \kern\dimen@ 
    \xtilde{\kern-\dimen@ \box\usefulbox\kern\dimen@ }\kern-\dimen@ }}%
 {}}

\newcommand*\xoverline[2][0.75]{%
    \sbox{\myboxA}{$\m@th#2$}%
    \setbox\myboxB\null
    \ht\myboxB=\ht\myboxA%
    \dp\myboxB=\dp\myboxA%
    \wd\myboxB=#1\wd\myboxA
    \sbox\myboxB{$\m@th\overline{\copy\myboxB}$}
    \setlength\mylenA{\the\wd\myboxA}
    \addtolength\mylenA{-\the\wd\myboxB}%
    \ifdim\wd\myboxB<\wd\myboxA%
       \rlap{\hskip 0.5\mylenA\usebox\myboxB}{\usebox\myboxA}%
    \else
        \hskip -0.5\mylenA\rlap{\usebox\myboxA}{\hskip 0.5\mylenA\usebox\myboxB}%
    \fi}

\def\xxoverline #1{\mathchoice
 {{\setbox\usefulbox=\hbox{$\m@th\displaystyle #1$}%
    \dimen@ \getslant\the\textfont\symletters \ht\usefulbox
    \divide\dimen@ \tw@ 
    \kern\dimen@ 
    \overline{\kern-\dimen@ \box\usefulbox\kern\dimen@ }\kern-\dimen@ }}
 {{\setbox\usefulbox=\hbox{$\m@th\textstyle #1$}%
    \dimen@ \getslant\the\textfont\symletters \ht\usefulbox
    \divide\dimen@ \tw@ 
    \kern\dimen@ 
    \xoverline{\kern-\dimen@ \box\usefulbox\kern\dimen@ }\kern-\dimen@ }}
 {{\setbox\usefulbox=\hbox{$\m@th\scriptstyle #1$}%
    \dimen@ \getslant\the\scriptfont\symletters \ht\usefulbox
    \divide\dimen@ \tw@ 
    \kern\dimen@ 
    \xoverline{\kern-\dimen@ \box\usefulbox\kern\dimen@ }\kern-\dimen@ }}
 {{\setbox\usefulbox=\hbox{$\m@th\scriptscriptstyle #1$}%
    \dimen@ \getslant\the\scriptscriptfont\symletters \ht\usefulbox
    \divide\dimen@ \tw@ 
    \kern\dimen@ 
    \xoverline{\kern-\dimen@ \box\usefulbox\kern\dimen@ }\kern-\dimen@ }}%
 {}}
\makeatother

\theoremstyle{definition}

\title{$0$-cycles and sheaves on abelian surfaces}
\author{Giovanni Mongardi}
\address{Giovanni Mongardi\\Alma Mater Studiorum, Università di Bologna, P.zza di porta san Donato 5, 40126 Bologna, Italia}
\email{giovanni.mongardi2@unibo.it}

\author{Gianluca Pacienza}
\address{Gianluca Pacienza\\ 
Universit\'e de Lorraine, CNRS, IECL\\
F-54000 Nancy -- France }
\email{gianluca.pacienza@univ-lorraine.fr}
\author{Laura Pertusi}
\address{Laura Pertusi\\
Università degli studi di Milano, via Cesare Saldini 50, 20133 Milano, Italia}
\email{laura.pertusi@unimi.it}

\begin{document}


\begin{abstract}
We introduce a filtration on the Chow ring of an abelian surface $A$, inspired by O'Grady's filtration on K3 surfaces. We give a geometric description of the filtration, and we prove that it is deeply linked with the rational orbit of points in the generalized Kummer variety of $A$. We propose a conjecture on the second Chern class of sheaves on this abelian surface, and we provide some evidence for the conjecture and prove some of its applications. 
\end{abstract}
\maketitle

\section{Introduction}
On a projective K3 surface $S$, Beauville and Voisin \cite{BV04} proved that there is a canonical zero cycle $o_S$, given by the class of any point on any rational curve on $S$, such that for every divisors $D,D'$ one has $D\cdot D'\in\Z o_S$. 
One can define a filtration in the Chow ring of zero cycles of $S$, where $S_g(S)$ is given by cycles equivalent to $x_1+\dots x_g+ko_S$, with $k\in \mathbb Z$. O'Grady \cite{OG13} proved that $S_g(S)$ consists of cycles supported on genus $g$ curves on $S$.  Building upon O'Grady's work and generalizing a result by Huybrechts \cite{Huy10}, Voisin \cite{Voi15} proved that the second Chern class of simple vector bundles falls into $S_d(S)$, where $d$ is half the dimension of the moduli space containing the vector bundle. By considering stable objects in $S$, Shen Yin and Zhao \cite{SYZ20} greatly generalized Voisin's result and proved that the second Chern class of a stable object lies in the prescribed piece of the above filtration.

The goal of our work is to introduce a filtration on the Chow group $\textrm{CH}_0(A)$ of $0$-cycles on an abelian surface $A$. To avoid ambiguities we denote by $a+_Ab$ the  sum, with respect to the group structure, of two points $a,b\in A$ on an abelian surface $A$  and by $a+b$ their sum in the Chow group $\textrm{CH}_0(A)$.  Inspired by the above works on projective $K3$ surfaces, for any positive integer $d$ we define the following subset of $\textrm{CH}_0(A)$ which will turn out to define an increasing filtration which is stable under multiplication by $\Z$ (cf.\ Lemma \ref{lem:filt} and Corollary \ref{cor:Z-inv}):
$$
S_d(A):=\{z\in \textrm{CH}_0(A):z+\iota(z)\equiv z'+2m o_A \textrm{ with } m\in \mathbb Z \textrm{ and $z'$ effective of degree $2d$}\},
$$
where $o_A\in A$ is the origin, $\iota$ the $(-1)$-involution and $\equiv$ the rational equivalence relation. 
For any integer $k\geq d$ we define inside $S_d(A)$ the following
$$
S^k_d(A):=\{z\in \textrm{CH}_0(A):z+\iota(z)\equiv z'+2(k-d) o_A \textrm{ with } m\in \mathbb Z \textrm{ and $z'$ effective of degree $2d$}\}.
$$
The filtration above can again be interpreted as cycles supported on curves (cf.\ Lemma \ref{lem:char-genus}).

Let $z\in A^{(d)}$. We denote by $O^\iota_z$ the set of all points $w\in A^{(d)}$ such that $$z+\iota(z)\equiv w+\iota(w)\in \textrm{CH}_0(A)$$ and by $O_z$ the set of all points $w\in A^{(d)}$ such that $z\equiv w\in \textrm{CH}_0(A)$.

\begin{theorem}\label{thm:orbits}
Let $A$ be an abelian surface and $k>d\geq 0$ two integers. Then
$$
S^k_d(A)=\{ z\in {\emph{CH}_0} (A): O^\iota_{z}\not=\emptyset \textrm{ and } \dim O^\iota_{z}\geq (k-d)\}.
$$
\end{theorem}

We then make the following conjecture:
\begin{conjecture}\label{conj:abelianVoisin}
 {\it Let $A$ be an abelian surface and $F$ a simple vector bundle on $A$ with Mukai vector $v:=v(F)$. Suppose that $c_1(F)$ is symmetric. Then
$
c_2(F)\in S_{d(v)-1}(A).
$}
\end{conjecture}
Recall that given a Mukai vector $v=(r,l,s)$, we set $d(v):=\frac{l^2}{2}-rs+1.$ This number coincides with $\frac{1}{2}\dim(\Ext^1(F,F))$ for any stable sheaf $F$ with Mukai vector $v$, that is half the dimension of the moduli space $M_H(v)$ (for any $v$ generic polarization $H$).

We provide partial evidence for this conjecture in Theorem \ref{thm:simple filtration}. In Section~\ref{sec:simpletostablesheaves}, we prove that the above conjecture implies the same for all symmetric stable torsion free sheaves (cf.\ Theorem \ref{thm:abelianVoisin-slopestablesheaves}).

An important consequence of the above conjecture concerns the existence of the coisotropic subvarieties as predicted by Voisin (cf.\ \cite{V16}). More precisely, denote by $K_H(v)$ the fiber of the Albanese morphism of the moduli space $M_H(v)$ of semistable sheaves with respect to a $v$-generic polarization $H$ on $A$ and Mukai vector $v$.

    \begin{theorem}\label{thm:coiso}
Let $A$ be an abelian surface and $v$ a primitive positive Mukai vector with $v^2 \geq 6$. Let $H$ be a $v$-generic polarization on $A$. 
 Assume that Conjecture~\ref{conj:abelianVoisin} holds. Then for any $0<i\leq d(v)-1$ there exists an algebraically coisotropic subvariety $Z\dashrightarrow B$ of codimension $i-1$ in $K_H(v)$ with constant cycle fibers.
\end{theorem}
While this paper was being prepared, Chen, Li and Zhang posted the very interesting work \cite{CLZ26} on the arXiv.  Among other things they introduce a similar framework to  study zero cycles on generalized Kummer type manifolds. 
We compare \cite[Conjecture 7.2]{CLZ26} with Conjecture \ref{conj:abelianVoisin} at the end of Section~\ref{sec:coisotropicsub}. 

\subsection*{Acknowledgements}
We would like to thank Kieran O'Grady and Xiaolei Zhao for many useful comments on a preliminary version of this work and Zaiyuan Chen, Zhiyuan Li, and Ruxuan Zhang for kind email exchange.

G.P. is partially supported by ANR project No. ANR-23-CE40-0026 (POK0).
L.P.\ is supported by the FIS 3 Starting Grant Stability Conditions on Noncommutative Varieties with Applications in Algebraic Geometry (FIS-2024-04715), and is a member of the Indam group GNSAGA.

\section{Symmetric curves and cycles on abelian surfaces}
In this section we prove that there are curves in any abelian surface which are symmetric and cover a curve of given genus $g$ on the quotient K3 surface $S$.
In the following, a $\iota$-invariant curve/cycle/line bundle
will be called symmetric. We denote with $\Pic^\iota(A)$ the group of symmetric line bundles. Notice that, by \cite[Lemma 4.6.2]{BL13}, for any $l\in \textrm{NS}(A)$ there exists a symmetric line bundle $L$ such that $c_1^{top}(L)=l$. 
\begin{proposition}\label{prop:symmetric curves}
    Let $A$ be an abelian surface and let $S$ be the minimal resolution of the quotient $A/\iota$. Then, for any given genus $g$, there exists an integral curve $C$ of geometric genus $g$ inside $S$ whose double cover $\widetilde{C}\subset A$ is a symmetric curve.
\end{proposition}
\begin{proof}
    It suffices to prove the existence of the curve $C$ on $S$, as the double cover will naturally be symmetric. In this setting, our claim is precisely \cite[Theorem A]{CGL}.    
\end{proof}
\begin{remark}\label{rem:hyperelliptic}
    An interesting case is when  $g=0$. We will call symmetric hyperelliptic all curves on $A$ covering an integral rational curve on $S$. The name is motivated by the fact that the curve is $\iota$-invariant and the action of $\iota$ on it coincides with a hyperelliptic involution.
    In particular, by the above proposition there exist ample and irreducible symmetric hyperelliptic curves. Moreover, as \cite[Theorem A]{CGL} gives a Zariski dense set of rational curves in $S$, it proves also the Zariski density of symmetric hyperelliptic curves on $A$. Furthermore, notice that for any point $p$ on a symmetric hyperelliptic curve we have $p+\iota(p)\equiv 2o_A$: indeed by definition the class of any element of the $g^1_2$ given by the hyperelliptic involution $\iota$ is equivalent to the pullback of the class of the point $o_S$, which is equal to $2o_A$.
\end{remark}

The following proposition is contained in \cite{Be86}, we provide here a different proof relying on the previous result:
\begin{proposition}\label{prop:BV}
The image of the intersection map restricted to symmetric classes 
$$
\Pic(A)^\iota \otimes Pic(A)^\iota \to \emph{CH}_0(A),\ L\otimes M\mapsto L\cdot M.
$$
has image in $2\Z o_A$. 
\end{proposition}
\begin{proof}
    As $\Pic(A)^\iota$ is generated by ample symmetric classes, it is sufficient to check that 
    \begin{equation}
        H\cdot H'\in \Z o_A
    \end{equation}
    for $H,H'$ ample symmetric line bundles. By Proposition \ref{prop:symmetric curves}, there exist $C\in |H|,\ C'\in |H'|$ symmetric hyperelliptic curves. In particular for any $p\in C\cap C'$ we have $\iota(p)\in C\cap C'$. Hence 
    $$
    C\cap C'=\sum_{i=1}^m p_i + \sum_{i=1}^m  \iota(p_i) 
    $$
    from which we deduce
    \begin{equation}
        H\cdot H'\equiv C\cdot C'\equiv 2m \cdot o_A
    \end{equation}
    by Remark \ref{rem:hyperelliptic}.
\end{proof}
 We now prove a result which will be useful further on. The analog for general $K3$ surfaces was proved in \cite[Theorem 1.2]{ML04} and its statement and proof inspired our result. 

\begin{theorem}\label{thm:maclean}
    Let $A$ be an abelian surface and $x\in A$ a general point. Then
    $$
    \{y\in A:y+\iota(y)\equiv x+\iota(x)\}
    $$
    is dense for the euclidean topology in $A$.
\end{theorem}
\begin{proof}
    We will be using \cite[Theorem 1.2 and its proof]{ML04}: the key point in Maclean's proof  is the existence of two distinct $1$-parameter families of irreducible elliptic curves on a general $K3$ surface. 
    We use \cite[Theorem A]{CGL}, which gives us the existence of (infinitely many) integral elliptic curves on $S$, which move in a one dimensional family by \cite[Proposition 2.9]{CGL}. 
    Hence Maclean's result \cite[Theorem 1.2]{ML04} holds on $S$, so that for any $p\in S$ general the locus $\{q\in S\,:\,q\equiv p\}$ is dense for the euclidean topology on $S$. It suffices now to take any preimage $y\in A$ of $p$ and by pullback we have
    $$\{\iota^*(p)\equiv \iota^*(q) \}=\{y\in A:y+\iota(y)\equiv x+\iota(x)\} $$
    is dense for the euclidean topology on $A$.
    
\end{proof}
\begin{remark}
    We remark that, using \cite[Theorem A and Proposition 2.9]{CGL} one can actually prove \cite[Theorem 1.2]{ML04} for any complex K3 surface $S$, so that for general $x\in S$
    the locus 
    $$ \{y\in S \text{ such that }y\equiv x\}$$
    is dense in the Euclidean topology.
\end{remark}

\section{The filtration on $\textrm{CH}_0$}
Following \cite{OG13} and \cite{Voi15}, we define the following filtration on the Chow ring of an abelian surface: 
\begin{definition}\label{def:filt}
Let $A$ be an abelian surface and $d\geq 0$ an integer. We set
$$
S_d(A):=\{z\in \textrm{CH}_0(A):z+\iota(z)\equiv y+2m\cdot o_A,\ m\in \Z,\ y \textrm{ effective $0$-cycle of degree } 2d \}.
$$
For any $k\geq d$ we will consider $S_d^k(A)\subset S_d(A)$ defined as follows
$$
S_d^k(A):=\{z\in \textrm{CH}_0(A):z+\iota(z)\equiv y+2(k-d)\cdot o_A,\ \ y \textrm{ effective $0$-cycle of degree } 2d \}.
$$
\end{definition}
Notice that, by construction, the cycle $y$ in the above definition is symmetric.

\begin{remark} \label{rem:additivity_filtration}
By definition we have that if $z \in S_d(A)$ and $z' \in S_{d'}(A)$, then $z+z' \in S_{d+d'}(A)$.    
\end{remark}

\begin{lemma}\label{lem:filt}
    $S_d(A)$ defines an increasing filtration $S_\bullet(A)$ on $\emph{CH}_0(A)$. 
\end{lemma}
\begin{proof}
Consider $z\in \textrm{CH}_0(A)$ such that $z+\iota(z)\equiv y+2m\cdot o_A$, $m\in \Z$, $y$ effective $0$-cycle of degree $2d$. Let $p=p_{d+1}$ be any point on a symmetric hyperelliptic curve on $A$, which exists by Proposition \ref{prop:symmetric curves}. Then $p+\iota(p)\equiv 2o_A$. Hence
$$
z+\iota(z)\equiv \big(y+p+\iota(p)\big) +2(m-1)\cdot o_A
$$ which shows the inclusion $S_d(A)\subset S_{d+1}(A),\ \forall d\geq 0$.
\end{proof}
We give a characterization of the filtration $S_\bullet(A)$ in terms of cycles supported on symmetric curves covering a curve of given genus. 

\begin{lemma}\label{lem:char-genus}
    For any $g\geq 0$ we have
    $$
S_g(A)=\bigcup\{f_*\emph{CH}_0 (\widetilde{C})\},$$
where $f\,:\widetilde{C}\to A$ is a non-constant map and $\widetilde{C}$ is a smooth curve such that $f(\widetilde{C})$ is symmetric and the induced involution $\iota$ on $\widetilde{C}$ is such that $C=\widetilde{C}/\iota$ has genus $g$.

\end{lemma}
\begin{proof}
We prove the inclusion $\supset$.   
Let $\widetilde{C}$ be a smooth curve and $f:\widetilde{C}\to A$ a non-constant map such that $f(\widetilde{C})$ is symmetric. By Proposition \ref{prop:symmetric curves}, such a curve exists and the quotient curve $C:
=\widetilde{C}/\iota$ 
of genus $g$ has a non constant map to the K3 surface $S=\widetilde{A}/\iota$. Let $\pi\,:\widetilde{C}\to C$ be the quotient map. 
Let $z\in f_*\textrm{CH}_0 (\widetilde{C})$. If $g=0$ by definition 
$z+\iota(z)=2\deg(z)o_A$, as the hyperelliptic involution coincides with $\iota$. We may then assume $g\geq 1$. 
Let $D$ be an ample symmetric hyperelliptic curve as provided by Remark \ref{rem:hyperelliptic}. By the invariance of $\widetilde{C}$ and $D$, for any $q\in \widetilde{C}\cap D$ we have $\iota(q)\in \widetilde{C}\cap D$, hence $q+\iota(q)\equiv 2o_A$. Consider the following $0$-cycle of degree $g$:
$$
\xi := z-(\deg(z)-g)q
$$
and let $\eta\in \textrm{CH}_0(C)$ be a cycle such that $\xi+\iota(\xi)=\pi^*(\eta)$. Notice that $\eta$ still has degree $g$.
By Riemann-Roch on $C$ we have 
\begin{equation}\label{eq:eff}
 h^0(\widetilde{C},\sO_{\widetilde{C}}(\xi+ \iota (\xi)))\geq h^0(C,\mathcal{O}_C(\eta))= \deg(\eta)+1-g=1> 0.   
\end{equation}
So, up to linear equivalence, we can suppose $\eta$ is effective and so is $\xi+\iota(\xi)$.
Hence
$z+\iota(z)$ can be written as 
$$
z+\iota(z)\equiv \big(\xi+\iota(\xi)\big) + (\deg(z)-g)q + (\deg(z)-g)\iota(q)\equiv \big(\xi+\iota(\xi)\big) + 2  (\deg(z)-g)o_A
$$
with $\xi+\iota(\xi)$ a degree $2g$ cycle which is effective by (\ref{eq:eff}).

Let us now prove the other inclusion: let $z\in S_g(A)$ and write $$z+\iota(z)=p_1+\dots p_g+\iota(p_1)+\dots +\iota(p_g)+2m\cdot o_A.$$ Let $q_i$ be the projection of $p_i$ to the K3 surface $S$\footnote{Notice that such a projection is not well defined if one of the points is of two torsion. However, in the crepant resolution $S\to A/\iota$ all points over the image of a two torsion point are rationally equivalent, hence we can pick any of them.}. By a parameter count, there exists a curve $C$ of genus at most $g$ passing through the points $q_1,\dots q_g$, and let $f\,:\widetilde{C}\to \iota^{-1}(C)$ be a resolution of singularities. By construction $f(\widetilde{C})$ contains the points $p_1,\dots p_g$ and $\iota(p_1),\dots,\iota(p_g)$, so that $z\in f_*CH_0(\widetilde{C})$.
\end{proof}
\begin{corollary}\label{cor:Z-inv}
For all $d\geq 0$ we have $\mathbb Z \cdot S_d(A)\subset S_d(A)$. 
\end{corollary}
\begin{proof}
    This follows from the elementary observation that $\mathbb Z \cdot \textrm{CH}_0 (C)\subset \textrm{CH}_0 (C)$ for any curve $C$ and from the previous Lemma \ref{lem:char-genus}. 
\end{proof}
\section{A characterization of the filtration via rational orbits}
Let $z\in A^{(d)}$. We denote by $O^\iota_z$ the set of all points $w\in A^{(d)}$ such that $z+\iota(z)\equiv w+\iota(w)\in \textrm{CH}_0(A)$.
We start with the following.
\begin{lemma}\label{lem:inc}
Let $A$ be an abelian surface and $k>d\geq 0$ two integers. Then
$$
S^k_d(A)\subset \{ z\in {\emph{CH}_0} (A): O^\iota_{z}\not=\emptyset \textrm{ and } \dim O^\iota_{z}\geq (k-d)\}.
$$
\end{lemma}
\begin{proof}
    Let $C\subset A$ be a symmetric hyperelliptic curve. For any $p\in C$, then $p+\iota(p)\equiv 2o_A$, and therefore for any $z\in C^{(k-d)}$, the cycle $z+\iota(z)$ is contained in the orbit of $2(k-d)o_A$. As a consequence, any $0$-cycle $z\in {\textrm{CH}_0} (A)$ such that $z+\iota(z)\equiv y+2(k-d)o_A$ with $y$ effective $0$-cycle of degree $k$ verifies $O^\iota_{z}\not=\emptyset \textrm{ and } \dim O^\iota_{z}\geq (k-d)$.
\end{proof}
We also record the two following results. 
\begin{lemma}\label{lem:2.2}
Let $A$ be an abelian surface
and $C\subset A$ a symmetric curve such that $\forall p,q\in C$ verify 
\begin{equation}\label{eq:equiv}
  p+  \iota(p)\equiv q+\iota(q)
\end{equation} in $A$. Then $\forall p\in C$ we have $p+\iota(p)\equiv 2o_A$.
\end{lemma}
\begin{proof}
    Take an ample symmetric hyperelliptic curve $D\subset A$ and consider $C\cap D$. For any $p\in C\cap D$ we have $\iota(p)\in C\cap D$, hence $p+\iota(p)\equiv 2o_A$ by Remark \ref{rem:hyperelliptic}, and the conclusion follows from (\ref{eq:equiv}). 
\end{proof}
\begin{lemma}\label{lem:2.3}
Let $A$ be an abelian surface. The union of curves $C$ satisfying the property stated in Lemma \ref{lem:2.2} is Zariski
dense in $A$.
\end{lemma}
\begin{proof}
    By \cite[Lemma 2.3]{Voi15} applied to the Kummer $K3$ surface $S$ associated to $A$, the  union of curves whose points are rationally
equivalent to the Beauville-Voisin cycle $c_S$ is dense. The pull-backs of these curves to $A$ yield the conclusion.
\end{proof}
We are now ready to prove the main result of this section. 
\begin{theorem}\label{thm:cni}
    Let $A$ be an abelian surface and $k\geq d\geq 0$ two integers. Let $Z\subset A^{(k)}$ an irreducible closed subvariety of dimension $k-d$ such that $\forall z,z'\in Z$ we have 
    \begin{equation}\label{eq:Z}
      z+\iota(z)\equiv z'+\iota(z')  
    \end{equation}
    Then any $z\in Z$ lies in $S^k_d(A)$. 
\end{theorem}
\begin{proof}
We argue as in  \cite[Proof of Theorem 2.1]{Voi15} and we proceed by induction on $k$. The case $k=d$ is trivial, and the case $k=1,\ d=0$ is provided by Lemma \ref{lem:2.2}. 
   We consider the $k$-th cartesian product $A^{\times k}=A\times \cdots \times A$ and denote by $p:=p_1$, respectively by $q:=p_2\times \cdots \times p_k$, the projection onto the first factor, respectively onto the last $(k-1)$-factors. Consider the natural morphism $\pi:A^{\times k}\to A^{(k)}$ and let $\tilde Z:=\pi^{-1}(Z)$. By abuse of notation we still denote by $p$ and $q$ the restrictions of these maps to $\tilde Z$.  
   We will distinguish two cases. 

   {\bf Case 1:} the map $p:\tilde Z\to A$ is dominant.

\noindent
Let $C\subset A$ be a curve such that $\forall p\in C$ we have $p+\iota(p)\equiv 2o_A$. Set $\tilde Z_C:=p^{-1}(C)$. 
Notice that points of $\tilde Z_C$ are of the form
\begin{equation}\label{eq:c,z'}
    (c,z'): c\in C,\ c+z'\in Z.
\end{equation}
In particular $c+\iota(c)\equiv 2o_A$.

We further distinguish two subcases. 

 {\bf Subcase 1.a:} $\dim (q(\tilde Z_C))=\dim (\tilde Z_C)=  k-d-1$.

From (\ref{eq:c,z'}), (\ref{eq:Z}) and the fact that $c+\iota(c)\equiv 2o_A$, for all $c\in C$ we deduce that
$$
z'+\iota(z')\equiv z+\iota(z)-2o_A
$$
does not depend on the choice of $z'\in q(\tilde Z_C)$.
In particular we can apply the inductive hypothesis to $q(\tilde Z_C)$ and deduce that 
$
z'\in S^{k-1}_d(A),$ i.e.
$$
z'+\iota(z')\equiv y'+2(k-1-d)o_A \textrm{ with $y'$ effective :  } \deg(y')=2d.
$$
Therefore for all $z\in Z_C$, $z=c+z'$ we have
$$
z+\iota(z)=\big(c+\iota(c)\big)+\big(z'+ \iota(z')\big)\equiv 2o_A + \big(y'+2(k-1-d)o_A  \big) = y'+ 2(k-d)o_A 
$$
with $y'$ effective such that $\deg(y')=2d$.

  {\bf Subcase 1.b:} $\dim (q(\tilde Z_C))<\dim (\tilde Z_C)=  k-d-1$, for any curve $C$ as above.

 \noindent
By Lemma \ref{lem:2.3}, curves $C$ such that $p+\iota(p)\equiv 2o_A,\ \forall p\in C$ are Zariski dense in $A$, therefore we can take one such curve and a point $z\in\tilde{Z}_C$ such that $z'$ is general both in $\tilde{Z}_C$ and $\tilde{Z}$, and both varieties are smooth in $z'$ of dimensions respectively $k-d-1$ and $k-d$. By hypothesis, the map $q$ is not of maximal rank at $z'$, which is general in $\tilde{Z}$. Hence, also $q(\tilde{Z})$ has dimension at most $k-d-1$. As all cycles $z$ parametrized by $\tilde{Z}$ have $z+\iota(z)$ constant in $\text{CH}_0(A)$, it follows that for any fiber $F$ of $q$, all points $c$ of $p(F)$ satisfy
$$ c+\iota(c)=2o_A.$$
However, this contradicts the fact that $p$ is surjective, as the general point $r$ of $A$ does not satisfy $r+\iota(r)=2o_A$.
 
   {\bf Case 2:} None of the projections onto the $j$-th factor, $j=1,\ldots k$, restricted to $\tilde Z$  is dominant.
Notice that if one of these projections $p_j$ were dominant up to exchanging the role of $p_1$ and $p_j$ we would be in {\bf Case 1}. 

   We set $C_j:=\textrm{Im}(p_j)$, if the image of the projection $p_j$ on the $j$-th
 factor is one-dimensional, while we pick any curve $C_j$ containing $\textrm{Im}(p_j)$ if this image is a point.
 Consider now an ample symmetric curve $C\subset A$ such that $ p+\iota(p)\equiv 2o_A,\ \forall p\in C$. The associated  line bundle 
 $L:=p_1^*\mathcal O_A(C)\otimes\cdots p_k^*\mathcal O_A(C)
$ is ample and so is its restriction to $C_1\times\cdots \times C_k$. 
One sees that the (non-empty) intersection between $\tilde Z$ and the $(k-d)$-self intersection of $L$ is 
$$
\tilde Z\cap \big[(k-d)!\sum_{i_1<\cdots <i_{k−d}} p_{i_1}^* C\cdot \ldots \cdot  p_{i_k}^* C\big]. 
$$
Therefore for some $i_1<\cdots <i_{k−d}$ we have 
$$ \tilde Z\cap p_{i_1}^* C\cdot \ldots \cdot p_{i_{k-d}}^* C\not= \emptyset. $$
An element in $\tilde Z\cap p_{i_1}^* C\cdots p_{i_{k-d}}^* C$ yields a cycle
$z=\xi+ w$ with $\xi\in C^{(k-d)}$ and $w\in A^{(d)}$. As a consequence
$$
z+\iota (z)=\big(\xi+\iota (\xi)\big)+\big(w+\iota (w)\big)\equiv 2(k-d)o_A+y$$
with $y:=w+\iota (w)$ effective $0$-cycle of degree $2d$ and we are done.
\end{proof}
\begin{proof}[Proof of Theorem \ref{thm:orbits}]
The result follows immediately from Lemma \ref{lem:inc} and Theorem \ref{thm:cni}. 
\end{proof}

\begin{proposition}\label{prop:rational orbit filtration}
    Let $z\in A^{(k)}$ be a point such that it is contained in the kernel of the summation map $A^{(k)}\to A$. Suppose that $z\in S^k_d(A)$, $k\neq d$. Then $O_z$ has dimension $k-d-1$.
\end{proposition}
\begin{proof}
As $z$ is in the kernel of the summation map, by \cite[Corollary 2.4]{Bl76} (see also \cite[Théorème]{Be86}) we have $z\equiv \iota(z)$. As $z\in S^k_d(A)$, we have
$$2z\equiv z+\iota(z)\equiv \eta+\iota(\eta)+2(k-d)o_A,$$
where $\eta$ is effective and of degree $d$. Up to replacing $\eta$ with a translate by a torsion point, we can suppose that also $\eta\in A^{(d)}$ is in the kernel of the summation map, hence $\iota(\eta)\equiv \eta$, from which we obtain 
that 
$$z\equiv \eta +(k-d)o_A.$$
Therefore, to prove our claim it suffices to prove that $\dim(O_{(k-d)o_A})=k-d-1$, which is the content of Lemma \ref{lem:rational orbit}.
\end{proof}
\begin{lemma}\label{lem:rational orbit}
    Let $A$ be an abelian surface. Then $\dim(O_{jo_A})=j-1$
\end{lemma}
\begin{proof}
    Our claim is the content of \cite[Section 2.2]{Lin16}. First, we suppose that $A$ is principally polarized and let $C$ be a symmetric theta divisor passign through $o_A$. We construct the variety $V_{j,1}$ as in loc.~cit. by taking the image of $C^{j+1}\to A^{(j)}$ and restricting to the kernel of the summation map. Here the first map is $(a_0,\dots,a_j,a)\mapsto (a_0+_Aa,\dots,a_j+_Aa)$. By \cite[Lemma 2.5]{Lin16}, the class of a point in $V_{j,1}$ is constant. Clearly, $jo_A\in V_{n,1}$, hence our claim holds for principally polarized abelian varieties. The Chow ring is preserved under isogenies of abelian varieties by \cite{Bl76} and all abelian surfaces are isogenous to a principally polarized abelian surface, hence our claim still holds.
\end{proof}

\begin{remark}\label{rem:dimension orbits}
    By definition $O_z\subset O_z^\iota$. Therefore, Theorem \ref{thm:orbits} implies a partial converse to Proposition \ref{prop:rational orbit filtration}: If we have an element $z\in A^{(k)}$ in the kernel of the summation map such that $\dim(O_z)\geq k-d-1$, we obtain that $z\in S^k_{d+1}(A)$. By Lemma \ref{lem:rational orbit}, we expect that in this situation one has $\dim(O_z^\iota)=\dim(O_z)+1$.
\end{remark}
In light of the above remark, we make the following:
\begin{conjecture}\label{conj:orbit dimension}
    Let $A$ be an abelian surface and let $z\subset A^{(k)}$ be a cycle in the kernel of the summation map. Then $\dim(O_z^\iota)=\dim(O_z)+1$.
\end{conjecture}
\section{The second Chern classes of simple vector bundles}

Given a Mukai vector $v=(r,l,s)$, we denote with $d(v):=\frac{l^2}{2}-rs+1.$ This number coincides with $\frac{1}{2}\dim(\Ext^1(F,F))$ for any stable sheaf $F$ with Mukai vector $v$, that is half the dimension of the moduli space $M_H(v)$ (for any $v$ generic polarization $H$). Note that $d(v)\geq 1$, since there are no spherical objects in $D^b(A)$ by \cite[Lemma 15.1]{Bri08}.
We study here Conjecture \ref{conj:abelianVoisin}.
To motivate the restrictions on the first and second Chern class, let us consider the following examples:
\begin{example}
    Let us take the Mukai vector $v=(1,0,-1)$. Elements in the moduli space of stable sheaves with Mukai vector $v$ can be described as the tensor product of the ideal sheaf $\sI_p$ of a single point $p\in A$ with an element $L\in\Pic^0(A)$. We have $d(v)=2$ and $c_2(\sI_p)=p$. Therefore  the ideal sheaf of a point lies in $S_1(A)$, or even in $S_0(A)$ if we have that $p+\iota(p)\equiv 2o_A$ (which by Theorem \ref{thm:maclean} holds on a dense subset of $A$). Tensoring with a degree zero line bundle $L$ changes the second Chern class by a multiple of $L^2$, which is an element of $S_0(A)$ if $L$ is symmetric, by Proposition \ref{prop:BV}. 
\end{example}
To provide some evidence for the conjecture, we have the following:
\begin{theorem}\label{thm:simple filtration}
Let $A$ be an abelian surface and $F$ a simple vector bundle of rank $r\geq 2$ on $A$ with Mukai vector $v:=v(F)$. Then
$$
c_2(F)\in S_{d(v)+r(r-2)}(A)
$$
\end{theorem}
\begin{remark}
    Notice that, if $r=2,$ this proves $c_2(F)\in S_{d(v)}(A)$ without any restriction on Chern classes. Therefore, assuming Conjecture \ref{conj:orbit dimension}, the above theorem actually gives $c_2(F)\in S_{d(v)-1}(A)$ for any sheaf with symmetric $c_1$ and $c_2$ which sums to zero, as in the proof of the theorem we obtain that $\dim (O_{c_2(F)})\geq k-d(v)$. 
\end{remark}
    As observed in \cite{OG13}, to prove the result for $F$ on a K3 surface it is sufficient to prove it for a twist $F\otimes L$ by a line bundle $L$. In the abelian case, the situation is analogous:
    \begin{lemma}\label{lem:global generation}
    Let $A$ be an abelian surface and let $F$ be a simple vector bundle on $A$ with Mukai vector $v:=v(F)$. Let $L$ be a symmetric line bundle.
    Then    
    $$c_2(F)\in S_j(A) \iff c_2(F\otimes L)\in S_j(A).$$
    \end{lemma}
    \begin{proof}
        By a direct computation with Chern polynomials, we have

        $$ c_2(F\otimes L)=c_2(F)+c_1(F)c_1(L)(\rk(F)-1)+\frac{1}{2}c_1(L)^2(\rk(F)^2-\rk(F)).$$
        By the assumption on $L$, we have that $c_1(L)^2\in S_0(A)$. If we consider $c_2(F\otimes L)+\iota(c_2(F\otimes L))$, we are left with the terms $c_2(F)+\iota(c_2(F))$ (which lies in $S_j(A)$ by hypothesis), $c_1(L)^2(\rk(F)^2-\rk(F))$ (which lies in $S_0(A)$ by the assumption on $L$ and $(c_1(F)+\iota c_1(F))c_1(L)(1-\rk(F))$. Now it suffices to notice that $c_1(F)+\iota(c_1(F))=c_1(\det(F)\otimes\iota(\det(F))$, which is symmetric, hence also the last term lies in $S_0(A)$, and the claim follows.
    \end{proof}

    Hence, up to tensoring with a sufficiently ample symmetric line bundle we may assume that $F$ is generated by its global sections and that $h^1(A,F^\vee)=0$.
    
We recall the following results from \cite{Voi15}. 
If $r:=\rk(F)$ and $W\subset H^0(A,F)$ is a general vector subspace of dimension $(r-1)$, we consider the evaluation map
$$
e_W:W\otimes \mathcal O_A\to F. 
$$
\begin{lemma}
    The morphism $e_W$ is generically injective, and the locus $Z \subset X$ where its
rank is $< r - 1$ consists of $k$ distinct reduced points, where $k = c^{top}_
2 (F)$.
\end{lemma}
Following \cite{Voi15}, the above lemma gives us a rational map
$\phi\,:G(r-1,H^0(A,F))\dashrightarrow A^{(k)}$ which sends a subspace $W$ to the subset $Z$ where $e_W$ is not injective.

\begin{proposition}[Proposition 3.2, \cite{Voi15}]\label{prop:simple rational map}
    If F is simple and satisfies $h^1(A,F^\vee)=h^1(A,F)=0$, the rational map $\phi$ is
generically one to one on its image.
\end{proposition}

With this result, we can now proceed to prove the main theorem of the section:
\begin{proof}[Proof of Theorem \ref{thm:simple filtration}]
    Let $F$ be a simple vector bundle and let $v(F):=(r,l,s)$, where $r=\rk(F)$, $l=c_1^{top}(F)$. By Lemma \ref{lem:global generation}, we can suppose that $F$ is globally generated and $H^1(A,F^\vee)=0.$ We now follow very closely \cite[Proof of Theorem 1.9]{Voi15}: we have that $\chi(A,\End(F))=(v,v^\vee)=2rs-l^2$ and, by Riemann-Roch, 
    $$\chi(A,\End(F))=(r-1)l^2-2rc_2^{top}(F).$$ 
    This gives $s=\frac{l^2}{2}-c_2^{top}(F).$ We set $k=c_2^{top}(F)$. 
      By a direct Chern class computation, we have that $\chi(A,F)=s$. As $F$ is globally generated and has trivial $H^1$, this number equals $h^0(A,F)$.
    Proposition \ref{prop:simple rational map} now gives a rational map from $G(r-1,h^0(A,F))$ to $A^{(k)}$ whose image is given by a constant cycle subvariety. 
    
    Its dimension is $(r-1)(s-r+1)=rs+2r-s-1-r^2$. By Theorem \ref{thm:cni}, it follows that $c_2(F)$ lies in $S^k_d(A)$, where $d=d(v)+r(r-2)$ and it is obtained by computing $k-d=\dim(G(r-1,s))$.
\end{proof}


\begin{corollary}
  Assume Conjecture \ref{conj:abelianVoisin}. Let $v\in H^*(A,\mathbb Z)$ be a Mukai vector.   Assume there exists a simple vector
bundle $F$ on $A$ with Mukai vector $v$. Then
$$
S^k_d(A)=\{c_2(G): G \textrm{ simple vector bundle on $A$ with Mukai vector $v$, whose $c_1$ is symmetric}\}$$
with 
$d=d(v)-1,\ k=c_2(v)= c_2^{top}(F)$.
\end{corollary}
\begin{proof}
    The inclusion ``$\supset$'' is the content of Conjecture~\ref{conj:abelianVoisin}. For the other inclusion we argue {\it mutatis mutandi} as in \cite[Proof of Corollary 3.4]{Voi15}. We repeat it here for greater clarity.
    Let 
    $$B:=\{c_2(G): G \textrm{ simple vector bundle on $A$ with $v(G)=v$ and symmetric $c_1(G)$}\}.$$
    Let $cl\,:A^{(d)}\to \text{CH}_0(A)$ be the cycle class map. We start by claiming that there is a dense open subset $U$ such that
    \begin{equation}\label{eq:B}
        cl(U)+cl(\iota(U))+2(k-d)o_A \subset B. 
    \end{equation}
    As $Y$ is simple and has symmetric $c_1$, we can find a smooth quasi projective variety $Y$, a point $y_0\in Y$ and a locally free sheaf $\sF$ on $Y\times A$ such that $F\cong \sF_{y_0}$, the Kodaira spencer map is injective and all sheaves parametrized by $Y$ are simple with symmetric $c_1$.
    Let $\Gamma:=c_2(\sF)\in \textrm{CH}^2(Y\times A)$ and consider the following set $R\subset Y\times A^{(d)}:$
    $$ R:=\{(y,z),\,\,\Gamma_*(y)+\iota(\Gamma_*(y))\equiv c_2(\sF_y)+\iota(c_2(\sF_y))\equiv cl(z)+\iota(cl(z))+2(k-d)o_A\in \text{CH}_0(A)\}, $$
where $k=c_2^{top}(F)$.
The first projection is surjective by Conjecture \ref{conj:abelianVoisin}. Let $R_0$ be a component of $R$ which dominates the first factor.
The  second projection $R_0\to A^{(d)}$  is also dominant. Indeed, using Mumford's theorem as in the case of $K3$ surface one sees that (\ref{eq:B}) holds. Then we conclude by invoking Theorem \ref{thm:maclean}.
\end{proof}

\section{From simple vector bundles to stable sheaves} \label{sec:simpletostablesheaves}

In this section  we explain the first consequence of Conjecture~\ref{conj:abelianVoisin}.

\begin{theorem} \label{thm:abelianVoisin-slopestablesheaves}
Let $A$ be an abelian surface and $H$ a fixed polarization. Assume that Conjecture~\ref{conj:abelianVoisin} holds. Then for every $H$-slope stable torsion free sheaf $F$ on $A$ such that $c_1(F)$ is symmetric, we have
$$c_2(F) \in S_{d(F)-1}(A),$$
where $d(F):=\frac{1}{2}\dim \Ext^1(F,F)$.
\end{theorem}

The argument follows the strategy of \cite[Proposition 1.7]{SYZ20} with some modifications to adapt it to this context. We first need the following two lemmas.

\begin{lemma} \label{lemma:c2oftriangles}
    Let
    \begin{equation}\label{eq:dist}
    E\to F \to G \to E[1]
    \end{equation}
    be a distinguished triangle in $D^b(A)$. If two of $c_2(E),\ c_2(F),\ c_2(G)$ lie in $S_i(A)$ and $S_j(A)$ respectively, then the third lies in $S_{i+j}(A)$.
\end{lemma}
\begin{proof} This is the analogue of \cite[Corollary 1.2]{SYZ20}.
From (\ref{eq:dist}) we deduce that 
$$c_2(F)= c_2(E) + c_2(G) + D,$$
where $D$ is spanned by intersections of divisor classes.
Taking the involution on both sides we get 
$$
c_2(F) + \iota (c_2(F)) = [c_2(E) + \iota(c_2(E))] + [c_2(G) + \iota(c_2(G))] + [D + \iota(D)]. 
$$
Note that $D + \iota(D)$ lies in the image of the intersection map restricted to symmetric classes 
$$
\Pic(A)^\iota \otimes Pic(A)^\iota \to \text{CH}_0(A),\ L\otimes M\mapsto L\cdot M.
$$
Hence the conclusion follows from Proposition~\ref{prop:BV} and Remark~\ref{rem:additivity_filtration}.
\end{proof}

\begin{lemma} \label{lemma:Mukairevisited}
Let
$$E\to F \to G \to E[1]$$
be a distinguished triangle in $D^b(A)$ such that $\Hom(E, G)=0$. If $c_2(E) \in S_{d(E)}(A)$ and $c_2(G) \in S_{d(G)-1}(A)$, then
$$c_2(F) \in S_{d(F)-1}(A).$$
\end{lemma}
\begin{proof}
By Mukai's Lemma \cite[Lemma 2.5]{BB17} we have 
$$d(E) + d(G) \leq d(F).$$
Combining the assumption, Lemma~\ref{lemma:c2oftriangles} and Lemma~\ref{lem:filt}, we conclude that $$c_2(F) \in S_{d(E)+d(G)-1}(A) \subset S_{d(F)-1}(A).$$
\end{proof}

\begin{proof}[Proof of Theorem~\ref{thm:abelianVoisin-slopestablesheaves}]
Consider the short exact sequence
$$0 \to F \to F^{\vee \vee} \to Q \to 0,$$
where $F^{\vee \vee}$ is a vector bundle and $Q$ is a torsion sheaf supported on points. Since $c_1(F)=c_1(F^{\vee \vee})$ is symmetric, Conjecture~\ref{conj:abelianVoisin} implies that $c_2(F^{\vee \vee}) \in S_{d(F^{\vee \vee})-1}$. On the other hand, if $d$ is the lenght of the support of $Q$, then $d(Q) \geq d$. Hence
$c_2(Q) \in S_d(A) \subset S_{d(Q)}(A)$.
Applying Lemma~\ref{lemma:Mukairevisited} to the distinguished triangle
$$Q[-1] \to F \to F^{\vee \vee} \to Q,$$
which satisfies $\Hom(Q[-1], F^{\vee \vee})=0$ since $F^{\vee \vee}$ is locally free and $Q$ is supported in dimension $0$, we deduce the statement.
\end{proof}

We point out the following corollary of Theorem~\ref{thm:abelianVoisin-slopestablesheaves}.

\begin{corollary}
Let $E$ be an iterated extension of a slope stable torsion free sheaf $F$ with $c_1(F)$ symmetric. If Conjecture~\ref{conj:abelianVoisin} holds, then
$$c_2(E) \in S_{d(E)-1}(A).$$
\end{corollary}
\begin{proof}
This is the analogue of \cite[Proposition 1.8]{SYZ20}. Note that $c_1(E)=mc_1(F)$ is symmetric and $c_2(E)=mc_2(F)+D$ for some positive integer $m$, where $D$ is the intersection of symmetric divisors. By Proposition~\ref{prop:BV} it follows that $D \in S_0(A)$. Thus by Theorem~\ref{thm:abelianVoisin-slopestablesheaves} and Corollary~\ref{cor:Z-inv} it follows that $c_2(E) \in S_{d(F)-1}(A)$. The conclusion follows from
$$2d(E)=v(E)^2+2\dim\Hom(E,E)=m^2v(F)^2+2\dim\Hom(E,E) \geq v(F)^2+2=2d(F).$$
\end{proof}

\begin{remark}
We expect that Theorem~\ref{thm:abelianVoisin-slopestablesheaves} holds for objects in $D^b(A)$ with symmetric first Chern class. However, the latter assumption does not allow us to use the induction argument as in \cite{SYZ20}.      
\end{remark}

\begin{remark}
Let $T$ be a torsion sheaf supported on a symmetric curve $\tilde{C}$. Then by Lemma~\ref{lem:char-genus} we have $c_2(T) \in S_g(A)$, where $g$ is the genus of $\tilde{C}/\iota$. On the other hand, the genus $\tilde{g}$ of $\tilde{C}$ satisfies
$$\tilde{g}=1+\tilde{C}^2 \leq \frac{1}{2}v(T)^2+ \dim\Hom(T,T) = d(T)$$
and
$$\tilde{g}=2g-1+\frac{B}{2}.$$
Thus
$$d(T)-1 \geq \tilde{g}-1= 2g-2+\frac{B}{2} \geq g.$$
The last equality follows from the fact that if $g< 2-\frac{B}{2}$, then $\tilde{g} \leq 1$ which is impossible on $A$. 
It follows that $c_2(T) \in S_{d(T)-1}(A)$. 
\end{remark}



\section{Coisotropic subvarieties} \label{sec:coisotropicsub}

Let $A$ be an abelian surface with a fixed polarization $H$. Let $v=(r,\ell, s) \in H^0(A, \Z) \oplus NS(A) \oplus H^4(A, \Z)$ be a primitive and positive Mukai vector, i.e.\ either $r>0$, or if $r=0$, $\ell$ is effective and $s \neq 0$, or $r=\ell=0$ and $s<0$. We denote by $M_H(v)$ the moduli space of $H$-semistable sheaves on $A$ with Chern class $v$, and let $M_H(v)^\text{st}$ be its stable locus. If $H$ is $v$-generic, then $M_H(v)=M_H(v)^\text{st}$. Let $\Phi \colon D^b(A)\to D^b(A^\vee)$ be the Fourier-Mukai transform associated to the Poincar\'e line bundle $\mathcal{P}$ on $A\times A^\vee$, and consider the morphism
$$alb_v\,: M_H(v) \to A\times A^\vee,\,\,F\mapsto (det(\Phi(F))\otimes det (\Phi(F_0^\vee)) ,det(F)\otimes det (F_0^\vee)),$$ 
where $F_0 \in M_H(v)$ is fixed. Denote by $K_H(v)$ a fiber of $alb_v$.

\begin{theorem}[\cite{Yo01}, Theorems 0.1 and 0.2]
Assume that $v^2 \geq 6$ and $H$ is $v$-generic. Then $alb_v$ is the Albanese morphism and $K_H(v)$ is a projective hyperk\"aheler manifold of dimension $v^2-2$ deformation equivalent to a generalized Kummer variety. 
\end{theorem}

An important consequence of Conjecture~\ref{conj:abelianVoisin} is the existence of algebraically coisotropic subvarieties of the form
$$
\xymatrix{
Y \ar@{^{(}->}[r] \ar@{-->}^p[d]& K_H(v) \\
B
}
$$
where $Y$ has codimension $i$ and the general fibers of $p$ are constant cycle subvarieties of $K_H(v)$ of dimension $i$.

Let us first make the following observations.

\begin{lemma} \label{lemma:c2sum0}
Assume that $v^2 \geq 6$ and $H$ is $v$-generic. Let $\alpha\,:M_H(v) \to A$ be the map sending a sheaf $F$ to the sum of $c_2(F)$. Then there exists a choice of $F_0\in M_H(v)$ such that $K_H(v)$ is contained in a fiber of $\alpha$.
\end{lemma}
\begin{proof}
By fixing a Mukai vector we fix the integer $s:=deg(c_2(\sF))$ for all sheaves $F\in M_H(v)$. We therefore have the actions of $A$ on $M_H(v)$ and on $A$ itself which are given by
$$ A\times M_H(v) \to M_H(v),\,\,(a,F) \mapsto t_a(F)$$
and
$$ A\times A\to A,\,(a,b)\mapsto (b+_A sa).$$
The morphism $\alpha$ is equivariant with respect to this action, hence $\alpha$ is an isotrivial fibration with isomorphic fibers. 
By universality of the Albanese map, all fibers of $alb_v$ are contained in fibers of $\alpha$ and, as the Albanese fibration is isotrivial by an analogous argument, we can choose a sheaf $F_0$ to ensure that $K_H(v)\subset \alpha^{-1}(0)$.
\end{proof}

\begin{proposition} \label{prop:MZabelian}
Let $E$, $F$ be two objects in $K_H(v)$. Then $[E]=[F]$ in $\emph{CH}_0(K_H(v))$ if and only if $c_2(E)=c_2(F)$.     
\end{proposition}
\begin{proof}
The proof given in \cite[ Theorem]{MZ20} works as well for moduli spaces on abelian surfaces. To pass to the fiber over zero of the Albanese morphism we observe that two rationally equivalent sheaves in the moduli space must lie in fibers differing by torsion points.  See also \cite[Theorem 6.10]{CLZ26}.    
\end{proof}

We are now ready to prove the main result of this section.

\begin{proof}[Proof of Theorem \ref{thm:coiso}]
We argue as in \cite[Theorem 0.5]{SYZ20}. Set $M:=M_H(v)$ and $d:=d(v)-1$.  We consider
the Hilbert scheme $A^{[d]}$ of $d$-points on $A$ and the following correspondence 
$$
     \overline R:=\{(\sE,\xi)\in M\times A^{[d]}: c_2(\sE)+i(c_2(\sE))\equiv \supp(\xi)+\iota(\supp(\xi))+2m\cdot o_A, \exists m\in \Z\}.
    $$
Set $K:= K_H(v)$ and $\textrm{Kum}:=\textrm{Kum}_{d-1}(A)$.
The main point is that by Lemma \ref{lemma:c2sum0} and \cite[Corollary 2.4]{Bl76} the correspondence induces the following correspondence    
    $$
     R:=\{(E,\xi)\in K\times \textrm{Kum}: c_2(E)\equiv \supp(\xi)+m\cdot o_A, \exists m\in \Z\}. 
    $$
    We denote by   $p_K$ and $p_{\textrm{Kum}}$ the projections of $R$ on the two factors. We start by analyzing the fibers of $p_K$ and $p_{\textrm{Kum}}$. By definition of $R$ it follows that for any $(E,\xi),(E,\xi')\in R$ we have 
        \begin{equation}\label{eq:p_M}
            \supp(\xi)\equiv \supp(\xi').
        \end{equation}
Analogously, if $(E,\xi),(E',\xi)\in R$, then $c_2(E) \equiv c_2(E')$. Proposition~\ref{prop:MZabelian} implies that $[E]\equiv [E'] \in \textrm{CH}_0(K)$.


Now by Theorem~\ref{thm:abelianVoisin-slopestablesheaves} the projection $p_K$ is dominant. Moreover, Conjecture~\ref{conj:abelianVoisin} impies that the projection $p_{\textrm{Kum}}$ is dominant, using that for every $\ell \in \textrm{NS}(A)$ we can find a symmetric line bundle corresponding to it.

Fix one component $R_0$ of $R$ dominating both factors. We restrict to the non-empty open subsets $U\subset K$ (respectively  $V\subset \textrm{Kum}$) over which the projection $p_K$ (resp. $p_{\textrm{Kum}}$) is surjective and finite. 

Consider a curve $C$ such that $\forall p\in C$ we have $p+\iota(p)\equiv 2o_A$. Let $V_{i,1}(C)$ be the subvariety of $\mathrm{Kum}$ associated with $C$ (using the construction \cite[Section 2.2]{Lin16}, see also the proof of Lemma~\ref{lem:rational orbit}). We set 
$$Z:=V_{i,1}(C)\times_{\mathrm{Kum}} R_0$$
and $\psi:Z\dashrightarrow R_0$ the natural induced map.
By Lemma \ref{lem:2.3}, as we let the curve $C$ vary among symmetric hyperelliptic curves, we may assume that $\psi(Z)\cap p_K^{-1}(U)\not=\emptyset$. Then the subset
$$
p_K(\psi(Z))
$$
parametrizes objects $\sE\in K$ such that $c_2(\sE)$ is constant in $\textrm{CH}_0(A)$.
\end{proof}

We finish this section with a comparison with \cite[Conjecture 7.2]{CLZ26}
\begin{proposition}
    Let $A$ be an abelian surface. Assume Conjecture \ref{conj:orbit dimension}. Then the following are equivalent:
    \begin{itemize}
        \item[(1)] Conjecture \ref{conj:abelianVoisin} holds.
        \item[(2)] Conjecture $7.2$ of \cite{CLZ26} holds.
    \end{itemize}
\end{proposition}
\begin{proof}
    Let $F$ be a sheaf in a moduli space $M_H(v)$. Let us suppose in addition that $F$ lies in a fiber of the Albanese map such that $c_1(F)$ is symmetric and $c_2(F)$ sums to zero. Then, by \cite[Corollary 2.4]{Bl76}, we have that $c_2(F)\equiv \iota (c_2(F))$, thus $c_2(F)$ and $c_2(F)_{(2)}$ (in the notation of \cite{CLZ26}) differ by a multiple of $o_A$, which does not change the stratum $S_i(A)$ of the filtration. Let $k$ be the degree of $c_2(F)$. By Theorem \ref{thm:orbits} and Proposition \ref{prop:MZabelian}, the rational equivalence orbit of $F$ in $K_H(v)$ is of dimension at least $k-d(v)$ if and only if $c_2(F)\in S_{d(v)-1}^{k}(A)$, which gives the equivalence between the two conjectures.

\end{proof}

\bibliography{literature}

\begin{thebibliography}{CGL22}

\bibitem[BB17]{BB17}
Arend Bayer and Tom Bridgeland.
\newblock Derived automorphism groups of {K}3 surfaces of {P}icard rank 1.
\newblock {\em Duke Math. J.}, 166(1):75--124, 2017.

\bibitem[Bea86]{Be86}
Arnaud Beauville.
\newblock Sur l'anneau de {C}how d'une vari\'et\'e ab\'elienne.
\newblock {\em Math. Ann.}, 273(4):647--651, 1986.

\bibitem[BL13]{BL13}
Christina Birkenhake and Herbert Lange.
\newblock {\em Complex abelian varieties}.
\newblock Springer Science \& Business Media, 2013.

\bibitem[Blo76]{Bl76}
Spencer Bloch.
\newblock Some elementary theorems about algebraic cycles on abelian varieties.
\newblock {\em Inventiones mathematicae}, 37(3):215--228, 1976.

\bibitem[Bri08]{Bri08}
Tom Bridgeland.
\newblock Stability conditions on {$K3$} surfaces.
\newblock {\em Duke Math. J.}, 141(2):241--291, 2008.

\bibitem[BV04]{BV04}
Arnaud Beauville and Claire Voisin.
\newblock On the {C}how ring of a {$K3$} surface.
\newblock {\em J. Algebraic Geom.}, 13(3):417--426, 2004.

\bibitem[CGL22]{CGL}
Xi~Chen, Frank Gounelas, and Christian Liedtke.
\newblock Curves on {K}3 surfaces.
\newblock {\em Duke Mathematical Journal}, 171(16):3283--3362, 2022.

\bibitem[CLZ26]{CLZ26}
Zaiyuan Chen, Zhiyuan Li, and Ruxuan Zhang.
\newblock Bloch's conjecture for equivalences between twisted abelian surfaces
  and applications.
\newblock {\em arXiv preprint arXiv:2606.22323}, 2026.

\bibitem[Huy10]{Huy10}
Daniel Huybrechts.
\newblock Chow groups of k3 surfaces and spherical objects.
\newblock {\em Journal of the European Mathematical Society}, 12(6):1533--1551,
  2010.

\bibitem[Lin16]{Lin16}
Hsueh-Yung Lin.
\newblock On the chow group of zero-cycles of a generalized kummer variety.
\newblock {\em Advances in Mathematics}, 298:448--472, 2016.

\bibitem[Mac04]{ML04}
Catriona Maclean.
\newblock Chow groups of surfaces with $h^{2,0}\leq 1$.
\newblock {\em Comptes Rendus. Math{\'e}matique}, 338(1):55--58, 2004.

\bibitem[MZ20]{MZ20}
Alina Marian and Xiaolei Zhao.
\newblock On the group of zero-cycles of holomorphic symplectic varieties.
\newblock {\em {\'E}pijournal de G{\'e}om{\'e}trie Alg{\'e}brique}, 4, 2020.

\bibitem[Oʼ13]{OG13}
Kieran~G OʼGrady.
\newblock Moduli of sheaves and the {C}how group of {K}3 surfaces.
\newblock {\em Journal de math{\'e}matiques pures et appliqu{\'e}es},
  100(5):701--718, 2013.

\bibitem[SYZ20]{SYZ20}
Junliang Shen, Qizheng Yin, and Xiaolei Zhao.
\newblock Derived categories of surfaces, o’grady’s filtration, and
  zero-cycles on holomorphic symplectic varieties.
\newblock {\em Compositio Mathematica}, 156(1):179--197, 2020.

\bibitem[Voi15]{Voi15}
Claire Voisin.
\newblock Rational equivalence of 0-cycles on k3 surfaces and conjectures of
  {H}uybrechts and {O’G}rady.
\newblock {\em Recent advances in algebraic geometry}, 417:422--36, 2015.

\bibitem[Voi16]{V16}
Claire Voisin.
\newblock Remarks and questions on coisotropic subvarieties and 0-cycles of
  hyper-{K}\"ahler varieties.
\newblock In {\em K3 surfaces and their moduli}, volume 315 of {\em Progr.
  Math.}, pages 365--399. Birkh\"auser/Springer, [Cham], 2016.

\bibitem[Yos01]{Yo01}
K{\=o}ta Yoshioka.
\newblock Moduli spaces of stable sheaves on abelian surfaces.
\newblock {\em Mathematische Annalen}, 321(4):817--884, 2001.

\end{thebibliography}
\bibliographystyle{alpha}

\end{document}